\documentclass[11pt, a4paper]{article}

\usepackage{lineno}

\usepackage{graphicx}
\usepackage{ae}
\usepackage{amsmath}
\usepackage{amssymb}
\usepackage{latexsym}
\usepackage{url}
\usepackage{epsfig}

\usepackage{cite}
\usepackage{mathrsfs}
\usepackage{amsfonts}
\usepackage{amsthm}

\def\qed{\hfill$\Box$\vspace{12pt}}

\usepackage{color}

\input amssym.def
\input amssym.tex

\long\def\delete#1{}

\newcommand{\be}{\begin{equation}}
\newcommand{\ee}{\end{equation}}
\newcommand{\bea}{\begin{eqnarray}}
\newcommand{\eea}{\end{eqnarray}}
\newcommand{\bean}{\begin{eqnarray*}}
\newcommand{\eean}{\end{eqnarray*}}

\def\deg{{\rm deg}}

\newtheorem{thm}{Theorem}[section]

\newtheorem{lem}[thm]{Lemma}
\newtheorem{prop}[thm]{Proposition}

\numberwithin{equation}{section}

\title{Laplacian spectral characterization of dumbbell graphs and theta graphs\thanks{Supported by the Natural Science Foundation of China (No.11361033).}}

\author{ Xiaogang Liu$^{1}$
\,and\, Pengli Lu$^{2,}$\thanks{Corresponding author.   E-mail
addresses: xiaogliu.yzhang@gmail.com (\textbf{X. Liu}), lupengli88@163.com (\textbf{P. Lu}).}
\\
\footnotesize{1. Department of Mathematics and Statistics, The University of Melbourne, Parkville, VIC 3010, Australia}\\
\footnotesize{2. School of Computer and Communication, Lanzhou University of Technology, Lanzhou, 730050, Gansu, P.R. China}
}
\date{}

\begin{document}

\openup 0.5\jot
\maketitle

\begin{abstract}
Let $P_n$ and $C_n$ denote the \emph{path} and \emph{cycle} on $n$ vertices respectively. The \emph{dumbbell graph}, denoted by $D_{p,k,q}$, is the graph obtained from two cycles $C_p$, $C_q$ and a path $P_{k+2}$ by identifying each pendant vertex of $P_{k+2}$ with a vertex of a cycle respectively. The \emph{theta graph}, denoted by $\Theta_{r,s,t}$, is the graph formed by joining two given vertices via three disjoint paths $P_{r}$, $P_{s}$ and $P_{t}$ respectively. In this paper, we prove that all dumbbell graphs as well as all theta graphs are determined by their Laplacian spectra.

\bigskip


\bigskip

\end{abstract}

\section{Introduction}

All graphs considered in this paper are simple and undirected. Let $G=(V(G),E(G))$ be a graph with vertex set $V(G)=\{v_1,v_2,\ldots,v_n\}$ and edge set $E(G)$. The \emph{adjacency matrix} of $G$, denoted by $A(G)$, is the $n \times n$ matrix whose $(i,j)$-entry is $1$ if $v_i$ and $v_j$ are adjacent and $0$ otherwise.  Let $d_i=d_i(G)=d_G(v_i)$ be the degree of the vertex $v_i$, and $\deg(G) = (d_1,d_2,\ldots,d_n)$ the degree sequence of $G$. We call
$L(G)=D(G)-A(G)$ (respectively, $Q(G)=D(G)+A(G)$) the \emph{Laplacian matrix} (respectively, \emph{signless Laplacian matrix}) of $G$, where $D(G)$ is the $n\times n$ diagonal matrix with $d_1,d_2,\ldots,d_n$ as diagonal entries.
Given an $n \times n$ matrix $M$, denote by $\phi(M;x)=\det(xI_n-M)$, or simply $\phi(M)$, the characteristic polynomial of $M$, where $I_n$ is the identity matrix of size $n$. The roots of the equation $\phi(M;x)=0$ are called the eigenvalues of $M$. We call the eigenvalues of $A(G)$, $L(G)$ and $Q(G)$ the \emph{adjacency eigenvalues}, \emph{Laplacian eigenvalues} and \emph{signless Laplacian eigenvalues} of $G$ respectively. The multiset of the eigenvalues of $A(G)$ is called the \emph{$A$-spectrum} of $G$.  Two graphs are said to be \emph{$A$-cospectral} if they have the same $A$-spectrum. A graph is called an \emph{$A$-DS graph} if it is \emph{determined by its $A$-spectrum}, meaning that there is no other non-isomorphic graphs $A$-cospectral with it. Similar terminology will be used for $L(G)$ and $Q(G)$, and the corresponding notations differ by a prefix ($A$-, $L$-, or $Q$-, respectively).

Which graphs are determined by their spectra? This is a classical question in spectral graph theory, which was raised by G\"{u}nthard and Primas \cite{kn:Gunthard56} in 1956 with motivations from chemistry. It is known that it is often very challenging to check whether a graph is determined by its spectrum or not, even for some simple-looking graphs. Although many graphs have been proved to be determined by their ($A$, $L$ or/and $Q$) spectra recently  \cite{kn:Boulet08, kn:Boulet09, kn:Brouwer12, kn:Bu12, kn:BuZhou12, kn:BuZhou13, kn:Cvetkovic95, kn:Cvetkovic10, kn:vanDam03, kn:vandam09, kn:Haemers08, kn:Ramezani09, kn:WangHuang09, kn:WangBelardo10, kn:WangHuang009, kn:Lu09, kn:LiuF12, kn:LiuF13, kn:LiuM10, kn:LiuM14, kn:Liu14, kn:Wangliu10, kn:Mirzakhah10, kn:Omidi07, kn:Zhou12, kn:Zhou13}, the problem of determining $A$-DS (respectively, $L$-DS, $Q$-DS) graphs is still far from being completely solved. Therefore, finding new families of DS graphs deserves further attention.

As usual, let $P_n$ and $C_n$ denote the \emph{path} and \emph{cycle} on $n$ vertices respectively. The \emph{dumbbell graph}, denoted by $D_{p,k,q}$, is the graph obtained from two cycles $C_p$, $C_q$ and a path $P_{k+2}$ by identifying each pendant vertex of $P_{k+2}$ with a vertex of a cycle respectively. The \emph{theta graph}, denoted by $\Theta_{r,s,t}$, is the graph formed by joining two given vertices via three disjoint paths $P_{r}$, $P_{s}$ and $P_{t}$ respectively (See Fig. \ref{DTeta1}). Due to the symmetry, in this paper we let $p\ge q\ge3$, $k\ge0$ in $D_{p,k,q}$ and $r\ge s\ge t\ge0$, $(s,t)\neq(0,0)$ in $\Theta_{r,s,t}$.
It is known \cite{kn:Ramezani09,kn:WangHuang009} that all theta graphs with no unique cycle $C_4$ are determined by their $A$-spectra. All dumbbell graphs $D_{p,k,q}$ without cycle $C_4$ satisfying $k\ge1$, except for $D_{3q,1,q}$, were proved to be determined by their $A$-spectra\cite{kn:WangHuang09,kn:WangBelardo10}. Moreover, all dumbbell graphs different from $D_{3q,0,q}$ and all theta graphs are determined by their $Q$-spectra \cite{kn:WangBelardo10}.

\begin{figure}
\centering
\includegraphics[height=7.6cm]{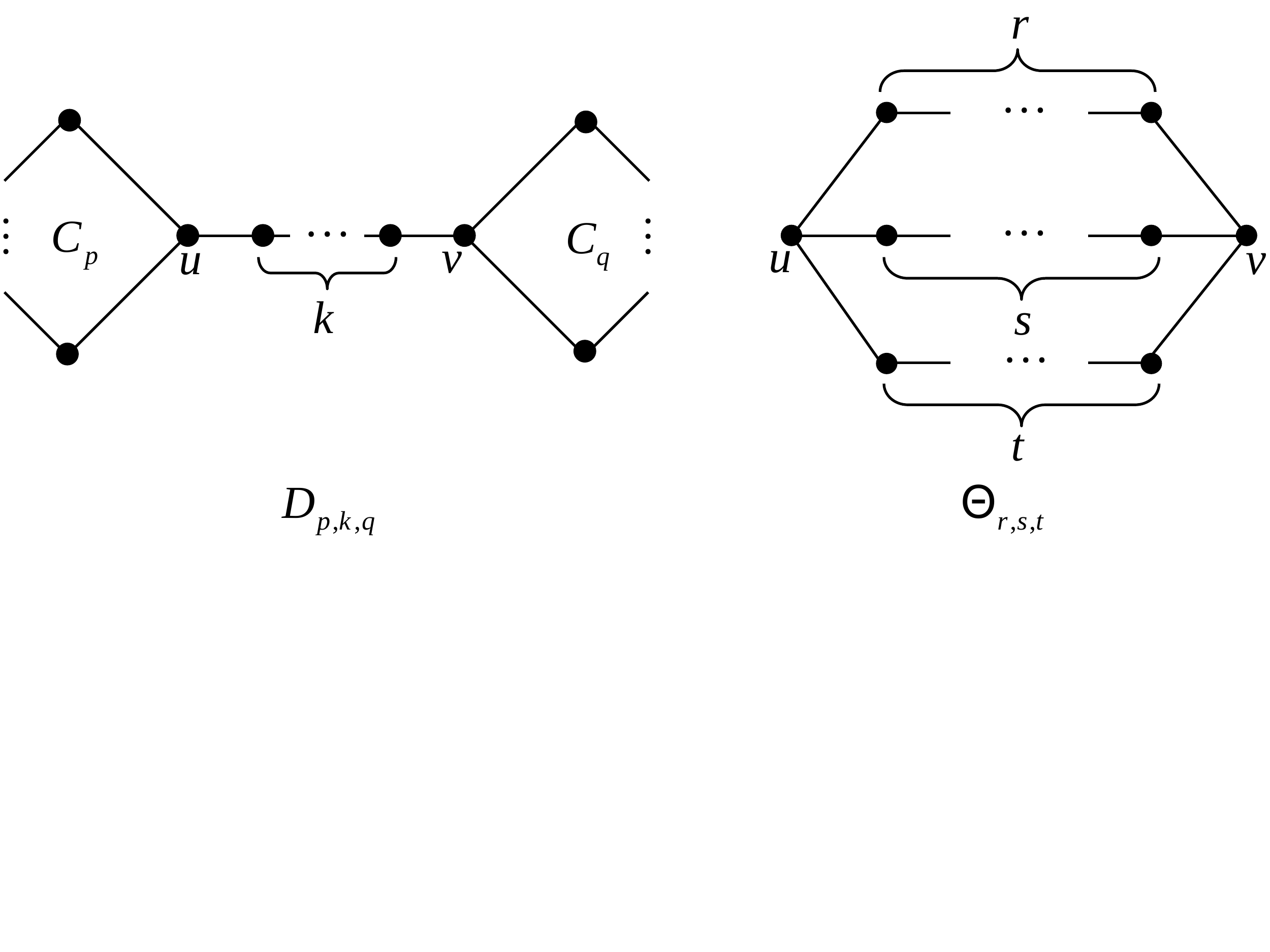}
\vspace{-3.46cm}
\caption{\small The graphs $D_{p,k,q}$ and $\Theta_{r,s,t}$.}
\label{DTeta1}
\end{figure}

Motivated by the results above, in this paper, we investigate the Laplacian spectral characterization of dumbbell graphs and theta graphs respectively. We prove that all dumbbell graphs as well as all theta graphs are determined by their Laplacian spectra. Our main results are as follows:
\begin{thm}
Every dumbbell graph is determined by its $L$-spectrum.
\end{thm}
\begin{thm}
Every theta graph is determined by its $L$-spectrum.
\end{thm}

\section{Preliminaries}

\begin{lem}\emph{\cite{kn:vanDam03,kn:Oliveira02}}\label{spectrum}
Let $G$ be a graph. The following can be determined by its $L$-spectrum:
\begin{itemize}
\item[\rm (a)] the number of vertices of $G$;
\item[\rm (b)] the number of edges of $G$;
\item[\rm (c)] the number of components of $G$;
\item[\rm (d)] the number of spanning trees of $G$;
\item[\rm (e)] the sum of squares of vertex degrees of $G$.
\end{itemize}
\end{lem}

Let $U_n$ be the matrix of order $n$ obtained from $L(P_{n+2})$ by deleting the rows and columns corresponding to the two end vertices of $P_{n+2}$.

\begin{lem}\label{LLpolynomial1}\emph{\cite{kn:Guo08}}
Set $\phi(L(P_0))=0$, $\phi(U_0)=1$. Then
\begin{itemize}
\item[\rm (a)] $\phi(L(P_{n+1}))=(x-2)\phi(L(P_n))-\phi(L(P_{n-1}))$;
\item[\rm (b)] $\phi(L(P_n))=x\phi(U_{n-1})$.
\end{itemize}
\end{lem}

Combining Lemma \ref{LLpolynomial1} and $\phi(L(P_1);x)=x$, we obtain the following formulas.

\begin{prop}\label{Leigenvalue4}
\begin{itemize}
\item[\rm (a)] $\phi(L(P_n);4)=4n;$  \quad \emph{\rm{(b)}}\, $\phi(U_{n};4)=n+1;$
\item[\rm (c)]$\phi(U_{n};2)=\frac{1}{2}\phi(L(P_{n+1});2)=\left\{\begin{array}{ll}
                                    0,& \text{if $n$ is odd;} \\
                                    (-1)^{n/2}, & \text{if $n$ is even.}
                                       \end{array}\right.$
\end{itemize}
\end{prop}

For a vertex $v$ of $G$, let $L_v(G)$ denote the principal sub-matrix of $L(G)$ formed by deleting the row and column corresponding to $v$. Similarly, if $H$ is a subgraph of $G$, let $L_H(G)$ denote the principal sub-matrix of $L(G)$
formed by deleting the rows and columns corresponding to all vertices of $V(H)$.

\begin{lem}\label{LCyclePol}\emph{\cite{kn:Guo13}}
Let $u$ be a vertex of $G$, $N(u)$ the set of vertices of $G$ adjacent to $u$, and $C(u)$ the set of cycles of $G$ containing $u$. Then
\begin{eqnarray*}
\phi(L(G);x)=(x-d_G(u))\phi(L_u(G);x)-\sum_{v\in N(u)}\phi(L_{uv}(G);x)-2\sum_{Z\in C(u)}(-1)^{|V(Z)|}\phi(L_Z(G);x).
\end{eqnarray*}
\end{lem}

The following formulas follow from Lemma \ref{LCyclePol}.
\begin{prop}\label{LpolDgraph}
Let $L_{q,k}=L_{C_p}(D_{p,k,q})$. Then
\begin{eqnarray*}
\phi(L_{q,k})&=&\big((x-3)\phi(U_{q-1})-2\phi(U_{q-2})-2(-1)^q\big)\phi(U_k)-\phi(U_{q-1})\phi(U_{k-1});\\
\phi(L(D_{p,k,q}))&=&\big((x-3)\phi(U_{p-1})-2\phi(U_{p-2})-2(-1)^p\big)\phi(L_{q,k})-\phi(U_{p-1})\phi(L_{q,k-1}).
\end{eqnarray*}
\end{prop}

\begin{prop}\label{LpolTheta}
Let $\theta_{r,s,t}=L_{u}(\Theta_{r,s,t})$. Then
\begin{eqnarray*}
   \phi(\theta_{r,s,t})&=&(x-3)\phi(U_r)\phi(U_s)\phi(U_t)-\phi(U_{r-1})\phi(U_s)\phi(U_t)-\phi(U_r)\phi(U_{s-1})\phi(U_t) \\
   & &-\phi(U_r)\phi(U_s)\phi(U_{t-1}); \\
   \phi(L(\Theta_{r,s,t})) &=&(x-3)\phi(\theta_{r,s,t})-\phi(\theta_{r-1,s,t})-\phi(\theta_{r,s-1,t})-\phi(\theta_{r,s,t-1})-2(-1)^{s+t}\phi(U_r) \\
   & &-2(-1)^{r+t}\phi(U_s)-2(-1)^{r+s}\phi(U_t).
\end{eqnarray*}
\end{prop}

Substituting (b) of Proposition \ref{Leigenvalue4} into Propositions \ref{LpolDgraph} and \ref{LpolTheta}, we obtain the following results.
\begin{prop}\label{LDTheta4}
\begin{itemize}
\item[\rm (a)] $\phi(L(D_{p,k,q});4)=4pqk-2p(2k+1)(1-(-1)^q)-2q(2k+1)(1-(-1)^p)+4(k+1)(1-(-1)^p)(1-(-1)^q)$;
\item[\rm (b)]$\phi(L(\Theta_{r,s,t});4)=4rst-2r(1+(-1)^{s+t})-2s(1+(-1)^{r+t})-2t(1+(-1)^{r+s})-2(1+(-1)^{s+t}+(-1)^{r+t}+(-1)^{r+s})$.
\end{itemize}
\end{prop}

In the following we compute the characteristic polynomials of $L(D_{p,k,q})$ and $L(\Theta_{r,s,t})$ respectively. It is known \cite{kn:Liu14} that
\begin{eqnarray}
          \phi(U_r)&=&\frac{y^{2r+2}-1}{y^{r+2}-y^{r}},\label{LaplacianPolHn}
        \end{eqnarray}
where $y$ satisfies the characteristic equation $y^2-(x-2)y+1=0$ with $x\neq4$.
Then plugging (\ref{LaplacianPolHn}) into Proposition \ref{LpolDgraph}, and with the help of Maple, we obtain
\begin{eqnarray}\label{LDumbwhole}
 y^{n}(y^2-1)^3\phi(L(D_{p,k,q}))+f(y)&=&f_D(p,k,q;y),
\end{eqnarray}
where $n=p+k+q$, $f(y)=1-2y-3y^2+4y^3+4y^4-4y^{2n+2}-4y^{2n+3}+3y^{2n+4}+2y^{2n+5}-y^{2n+6}$, and
\begin{eqnarray*}
\begin{array}{llll}
f_D(p,k,q;y)=
&y^{6+2q+2k}                          &+2(-1)^{1+p}y^{p+2q+2k+3}            &+4(-1)^{1+p}y^{p+2q+2k+2}\\
&+6(-1)^py^{p+4+2q+2k}                &+2(-1)^{1+p}y^{p+6+2q+2k}            &+2(-1)^py^{p+5+2q+2k}   \\
&+2(-1)^py^{p+2q+1}                   &+2(-1)^py^{p+2k+3}                   &+2(-1)^{1+p}y^{p+2+2q}    \\
&+2(-1)^{1+p}y^{p+6+2k}               &+2(-1)^{1+p}y^{p+5+2k}               &+2(-1)^{1+p}y^{1+p}    \\
&+6(-1)^{1+p}y^{p+2}                  &+2(-1)^{1+p}y^{p+3+2q}               &+4(-1)^py^{p+4}      \\
&+2(-1)^py^{p+3}                      &+2(-1)^py^{p+2q}                     &-2y^{2p+3+2k}    \\
&-2y^{2p+1+2q}                        &-3y^{2p+4+2k}                        &+y^{2p+6+2k}     \\
&-y^{2p+2+2q}                         &-2y^{2q+2k+3}                        &-3y^{2q+2k+4}  \\
&+3y^{2p+2}                           &+2y^{5+2k}                           &+y^{6+2k}    \\
&+2y^{3+2q}                           &-y^{2p+2q}                           &+2y^{2p+3}   \\                  &\textcolor[rgb]{0.00,0.00,1.00}{+y^{2k+4}}  &+3y^{2q+2}                    &+2(-1)^py^p \\
&+2(-1)^py^{p+4+2k} &\textcolor[rgb]{0.00,0.00,1.00}{+2(-1)^qy^q}           &-y^{2q}   \\
&+4(-1)^{p+1+q}y^{p+4+q}              &+8(-1)^{p+q}y^{p+2+q}                &-y^{2p}   \\                     &+2(-1)^{1+q}y^{2p+3+q+2k}            &+4(-1)^{1+q}y^{2p+q+2k+2}            &+2(-1)^{1+q}y^{2p+6+q+2k}      \\
&+6(-1)^qy^{2p+4+q+2k}                &+2(-1)^qy^{2p+5+q+2k}                &+6(-1)^{1+q}y^{q+2}            \\
&+4(-1)^qy^{4+q}                      &+2(-1)^qy^{q+3}                      &+2(-1)^{1+q}y^{1+q}            \\
&+8(-1)^{p+1+q}y^{p+4+q+2k}           &+4(-1)^{p+q}y^{p+q+2k+2}             &+2(-1)^qy^{q+2k+3}             \\
&+4(-1)^{p+1+q}y^{p+q}                &+4(-1)^{p+q}y^{p+6+q+2k}             &+2(-1)^{1+q}y^{2p+3+q}         \\
&+2(-1)^qy^{2p+1+q}                   &+2(-1)^{1+q}y^{6+q+2k}               &+2(-1)^{1+q}y^{q+2k+5}         \\
&+2(-1)^qy^{q+2k+4}                   &+2(-1)^{1+q}y^{2p+2+q}               &+2(-1)^qy^{2p+q}.
\end{array}
\end{eqnarray*}

Similarly, plugging (\ref{LaplacianPolHn}) into Proposition \ref{LpolTheta}, and with the help of Maple, we obtain
\begin{eqnarray}\label{LThetawhole}
 y^{n'}(y^2-1)^3\phi(L(\Theta_{r,s,t}))+f'(y)&=&f_\Theta(r,s,t;y),
\end{eqnarray}
where $n'=r+s+t+2$, $f'(y)=1-2y-3y^2+4y^3+4y^4-4y^{2n'+2}-4y^{2n'+3}+3y^{2n'+4}+2y^{2n'+5}-y^{2n'+6}$, and
\begin{eqnarray*}
\begin{array}{llll}
f_\Theta(r,s,t;y)=
&y^{2t+6+2r}                    &+y^{4+2s}                                   &-y^{2r+4+2s} \\
&+2(-1)^{r+s}y^{2+r+s}          &-2y^{2t+5+2r}                               &-2y^{2r+2s+5} \\
&-2y^{2t+5+2s}                  &+2(-1)^{t+1+s}y^{4+s+t+2r}                  &+2(-1)^{t+1+r}y^{4+r+t+2s} \\ &+2(-1)^{r+s+1}y^{4+r+s+2t}     &-y^{6+2t+2s}                                &+4(-1)^{r+s}y^{6+r+s+2t}\\       &+4(-1)^{r+s+1}y^{4+r+s}        &+4(-1)^{t+1+r}y^{4+r+t}                     &+2(-1)^{r+s+1}y^{8+r+s+2t}\\     &+2(-1)^{r+s}y^{6+r+s}          &+2(-1)^{t+1+r}y^{8+r+t+2s}                  &+2(-1)^{r+t}y^{6+r+t} \\         &+4(-1)^{r+t}y^{6+r+t+2s}       &+2(-1)^{s+t}y^{6+s+t}                       &+4(-1)^{s+t}y^{6+s+t+2r} \\      &+4(-1)^{t+1+s}y^{4+s+t}        &+2(-1)^{t+1+s}y^{8+s+t+2r}                  &+2y^{5+2s}  \\
&+2y^{5+2r}                     &-y^{2t+4+2s}                                &-y^{2t+4+2r} \\                  &\textcolor[rgb]{0.00,0.00,1.00}{+y^{2t+4}}  &+y^{2s+6}                      &+y^{2r+6}    \\
&+y^{2t+6}                                   &-y^{6+2r+2s}                   &+y^{2r+4} \\
&+2y^{5+2t}                                  &+2(-1)^{r+t}y^{2+r+t}       &\textcolor[rgb]{0.00,0.00,1.00}{+2(-1)^{s+t}y^{2+s+t}}.
\end{array}
\end{eqnarray*}

\section{Main results}

\begin{lem}\label{LDumb}
No two non-isomorphic dumbbell graphs are $L$-cospectral.
\end{lem}

\begin{proof}
Let $D_{p,k,q}$ and $D_{p',k',q'}$ be $L$-cospectral dumbbell graphs with $n=p+q+k$ and $n'=p'+q'+k'$ vertices, respectively. Without loss of generality, we let $p\geq q$ and $p'\geq q'$. By (a) and (d) of Lemma \ref{spectrum}, we have
\begin{eqnarray}
p+q+k&=&p'+q'+k'.\label{LDumbE01}\\
pq&=&p'q'.\label{LDumbE02}
\end{eqnarray}
By (\ref{LDumbwhole}), we then get
\begin{eqnarray}\label{LemDumb1}
f_D(p,k,q;y)=f_D(p',k',q';y).
\end{eqnarray}
Clearly, the term in $f_D(p,k,q;y)$ with the smallest exponent is $2(-1)^qy^{q}$ or $y^{2k+4}$ (written in blue), and similarly for $f_D(p',k',q';y)$. From (\ref{LemDumb1}) we have either $2(-1)^qy^{q}=2(-1)^{q'}y^{q'}$ or $y^{2k+4}=y^{2k'+4}$. In the former case, we have $q=q'$, and so $p=p'$ and $k=k'$ by (\ref{LDumbE01}) and (\ref{LDumbE02}). In the latter case, we have $k=k'$, and so $(p,q)=(p',q')$ by (\ref{LDumbE01}) and (\ref{LDumbE02}). Therefore, $D_{p,k,q}$ and $D_{p',k',q'}$ are isomorphic in each case.  \qed\end{proof}

\begin{lem}\label{LTheta}
No two non-isomorphic theta graphs are $L$-cospectral.
\end{lem}

\begin{proof}
Let $\Theta_{r,s,t}$ and $\Theta_{r',s',t'}$ be $L$-cospectral theta graphs with $n=r+s+t+2$ and $n'=r'+s'+t'+2$ vertices, respectively. Without loss of generality, we let $r\geq s\ge t$ and $r'\geq s'\ge t'$. By (a) and (d) of Lemma \ref{spectrum}, we have
\begin{eqnarray}
r+s+t&=&r'+s'+t'.\label{LThetaE01}\\
rs+rt+st&=&r's'+r't'+s't'.\label{LThetE02}
\end{eqnarray}
By (\ref{LThetawhole}), we then get
\begin{eqnarray}\label{LemTheta1}
f_\Theta(r,s,t;y)=f_\Theta(r',s',t';y).
\end{eqnarray}
Clearly, the term in $f_\Theta(r,s,t;y)$ with the smallest exponent is $2(-1)^{s+t}y^{2+s+t}$ or $y^{2t+4}$ (written in blue), and similarly for $f_\Theta(r',s',t';y)$. From (\ref{LemTheta1}) we have either $2(-1)^{s+t}y^{2+s+t}=2(-1)^{s'+t'}y^{2+s'+t'}$ or $y^{2t+4}=y^{2t'+4}$. In the former case, we have $s+t=s'+t'$, and so $r=r'$ by (\ref{LThetaE01}). This implies that $rs+rt=r's'+r't'$. By (\ref{LThetE02}), we then get $st=s't'$, which together with $s+t=s'+t'$ implies that $s=s'$ and $t=t'$. In the latter case, we have $t=t'$, and so $(r,s)=(r',s')$ by (\ref{LThetaE01}) and (\ref{LThetE02}). Therefore, $\Theta_{r,s,t}$ and $\Theta_{r',s',t'}$ are isomorphic in each case.  \qed\end{proof}

\begin{lem}\label{LDumbTheta}
There is no dumbbell graph $L$-cospectral with a theta graph.
\end{lem}
\begin{proof}
Let $D_{p,k,q}$ and $\Theta_{r,s,t}$ be $L$-cospectral theta graphs with $n=p+q+k$ and $n'=r+s+t+2$ vertices, respectively. Without loss of generality, we let $p\geq q\ge3$ and $r\geq s\ge t\ge0$. By (a) and (d) of Lemma \ref{spectrum}, we have
\begin{eqnarray}
p+q+k&=&r+s+t+2.\label{LDumbThetaE01}\\
pq&=&(r+1)(s+1)+(r+1)(t+1)+(s+1)(t+1).\label{LDumbThetaE02}
\end{eqnarray}
By (\ref{LDumbwhole}) and (\ref{LThetawhole}), we then get
\begin{eqnarray}\label{LemDumbTheta1}
f_D(p,k,q;y)=f_\Theta(r,s,t;y).
\end{eqnarray}
The term in $f_D(p,k,q;y)$ with the smallest exponent is $2(-1)^qy^{q}$ or $y^{2k+4}$, and the term in $f_\Theta(r,s,t;y)$ with the smallest exponent is $2(-1)^{s+t}y^{2+s+t}$ or $y^{2t+4}$. From (\ref{LemDumbTheta1}) we have either $2(-1)^qy^{q}=2(-1)^{s+t}y^{2+s+t}$ or $y^{2k+4}=y^{2t+4}$. In the former case, we have $q=s+t+2$, and so $p=r-k$ by (\ref{LDumbThetaE01}). Plugging $p$ and $q$ into (\ref{LDumbThetaE02}), we then get $st+(s+t)(k+2)+2k+3=0$, a contradiction. In the latter case, we have $k=t$. If $k=t\ge1$, plugging (\ref{LDumbThetaE01}) and (\ref{LDumbThetaE02}) into Proposition \ref{LDTheta4}, we then have $\phi(L(D_{p,k,q});4)\ge 4pqt-4(2t+1)(p+q)+16(t+1) =4rst+4(t+1)((p+q)(t-1)+3)+4 >4rst\ge \phi(L(\Theta_{r,s,t});4)$, a contradiction. If $k=t=0$, then (\ref{LemDumbTheta1}) implies that
\begin{eqnarray}\label{LemDumbTheta2}
f_D(p,0,q;y)-y^4-2y^5-y^6=f_\Theta(r,s,0;y)-y^4-2y^5-y^6.
\end{eqnarray}
Now, the term in $f_D(p,0,q;y)-y^4-2y^5-y^6$ with the smallest exponent is $2(-1)^qy^{q}$, and the term in $f_\Theta(r,s,0;y)-y^4-2y^5-y^6$ with the smallest exponent is $2(-1)^{s}y^{2+s}$. From (\ref{LemDumbTheta2}), we have $q=s+2$, and so $p=r$ by (\ref{LDumbThetaE01}). Then $pq=rs+2r<(r+1)(s+1)+(r+1)+(s+1)$, a contradiction to (\ref{LDumbThetaE02}). This completes the proof.
\qed\end{proof}

\begin{lem}\label{DTSETS}
Let $\mathscr{B}$ be the set of all dumbbell graphs and all theta graphs. If $G\in\mathscr{B}$ and $G'$ is $L$-cospectral with $G$, then $G'\in\mathscr{B}$.
\end{lem}
\begin{proof}
Suppose $G$ has $n$ vertices. Then the degree sequence $\deg(G)=(3^2,2^{n-2})$, where the exponent denotes the number of vertices in $G$ having the corresponding degree. Note that $G'$ is connected by (c) of Lemma \ref{spectrum}. Denote by $x_i$ the number of vertices with degree $i$ in $G'$ for $i=1,2,\ldots$. Then (a), (b) and (e) of Lemma \ref{spectrum} imply that
\begin{eqnarray}
& &  \sum_{i\ge1}x_i=n,  \label{DTSUM1}\\
& &  \sum_{i\ge1}ix_i=2(n+1),  \label{DTSUM2}\\
& & \sum_{i\ge1}i^2x_i=2\times9+4(n-2)=4n+10.\label{DTSUM3}
\end{eqnarray}
Combining the equations above, we get $$\sum_{i\ge1}(i^2-3i+2)x_i=4,$$
which implies that the largest degree of $G'$ is $3$ and $x_3=2$. By (\ref{DTSUM1}) and (\ref{DTSUM2}), we then have $x_1=0$ and $x_2=n-2$. Thus, $G'\in\mathscr{B}$.
\qed\end{proof}

Combining Lemmas \ref{LDumb}--\ref{DTSETS}, we conclude that
\begin{thm}
Every dumbbell graph as well as every theta graph is determined by its $L$-spectrum.
\end{thm}

\end{document}